\documentclass[reqno,12pt]{amsart}
\usepackage{amsmath,amssymb,epsfig,color,url}

\numberwithin{equation}{section}

\newtheorem{theorem}{Theorem}[section]

\newtheorem{lemma}[theorem]{Lemma}

\newtheorem{definition}[theorem]{Definition}
\newtheorem{corollary}[theorem]{Corollary}

\theoremstyle{remark}

\newtheorem{remark}[theorem]{Remark}

\newcommand{\di}{\displaystyle}

\newcommand{\vanish}[1]{}

\begin{document}

\title[Hurwitzian continued fractions]{Hurwitzian continued fractions
  containing a repeated constant and an arithmetic progression}

\author[G\'abor Hetyei]{G\'abor Hetyei}

\address{Department of Mathematics and Statistics,
  UNC-Charlotte, Charlotte NC 28223-0001.
WWW: \tt http://www.math.uncc.edu/\~{}ghetyei/.}



\date{\today}

\begin{abstract}
We prove an explicit formula for infinitely many convergents of 
Hurwitzian continued fractions that repeat several copies of the same
constant and elements of one arithmetic progression, in a quasi-periodic
fashion. The proof involves combinatorics and formal Laurent
series. Using very little analysis we can express their limits 
in terms of (modified) Bessel functions and Fibonacci polynomials.
The limit formula is a generalization of Lehmer's theorem that implies the
continuous fraction expansions of $e$ and $\tan(1)$, and it can also be
derived from Lehmer's work using Fibonacci 
polynomial identities. We completely characterize those implementations
of our limit formula for which the parameter of each Bessel function is
the half of an odd integer, allowing them to be replaced with elementary
functions.   
\end{abstract}

\maketitle

\section*{Introduction}

It is a remarkable property of infinite continued fractions that they often
define a sequence whose limit is easier to describe than the
individual entries. Even for $[1,1,\ldots]$, the
simplest example, the limit is easily found by solving a quadratic
equation, but it takes a bit longer to find a formula for the
convergents, in terms of Fibonacci numbers. 

The subclass of {\em Hurwitzian} (see Section~\ref{ss:hurwitz})
continued fractions of the form 
$$[\underbrace{\alpha,\ldots,\alpha}_r,\beta_0,
\overline{\underbrace{\alpha,\ldots,\alpha,\beta_0+\beta_1\cdot
    n}_d}]_{n=1}^\infty$$ 
that we propose investigating seems to be no different in this
regard. For the special case $d=1$ and $r=0$, D.\ H.\ Lehmer
\cite{DHLehmer} proved a formula for the limit in terms of Bessel
functions, that can be verified easily, after having guessed the
correct answer.  
Using another result from \cite{DHLehmer} and some
Fibonacci polynomial identities, it is not hard to generalize Lehmer's
formula to Hurwitzian continued fractions of the above form (see
Remark~\ref{rem:Lehmersway}). The resulting generalization has several
famous special cases, the most famous ones being Euler's formula for $e$
and the formula for $\tan(1)$. On the other hand, the only research
regarding the convergents themselves seems to be the
work~\cite{DNLehmer} of D.\ N.\ Lehmer (D.\ H.\ Lehmer's father!), who
proved congruences for their numerators and denominators by induction,
in somewhat greater generality, but without stating the values of the
convergents in an explicit fashion.  

Our work contains such an explicit formula, stated in
Section~\ref{sec:main} and proved in Section~\ref{sec:conv}. A variant
of the Euler-Mindig formulas using shifted partial denominators
yields a summation formula with many vanishing terms. This leads to a compact 
recurrence for the numerators of the $(nd+r-1)$st convergents that may
be restated as a linear differential equation for a formal Laurent
series. Classical textbooks on differential equations instruct us to
solve the associated homogeneous equation first and then find the
general solution by ``variation of parameters''. Unfortunately, in our
case the ``solution'' of the homogeneous equation turns out to be a a
two-way infinite formal sum. After discarding
infinitely many terms to have a formal Laurent series, the
``spirit'' of the classical method  still inspires a good guess for a
form of the solution, where the transformed differential equation
encodes a recurrence that is easily solved by inspection. Returning to
the original Laurent series involves using two polynomial summation
formulas that can be shown purely combinatorially, and seem to be
interesting by their own right. These formulas are shown in
Section~\ref{sec:lemma}. As outlined in Section~\ref{sec:main}, our
explicit formulas for the $(nd+r-1)$st convergents allow us to calculate the
limits using very little analysis. To find only the limits, this
approach is a bit more tedious than the one outlined in Lehmer's
work~\cite{DHLehmer}, but we gain a little more insight by also
obtaining asymptotic formulas for the numerators and the denominators of
the convergents.  

Section~\ref{sec:examples} is motivated by Komatsu's recent
remark~\cite{Komatsu-longp} stating that all known examples of
Hurwitzian continued fractions seem to have a short quasi-period and
involve (hyperbolic) trigonometric functions. For our class of
Hurwitzian continued fractions, the limit may be expressed in terms of
(modified) Bessel functions, which are only known to have an elementary form
when their order is the half of an odd integer, and then they are 
rational expressions of radical and (hyperbolic) trigonometric functions. 
We describe all such continued fractions in our class, and find that
their quasi-period $d$ can not be longer than 
$3$. This result illustrates some of the obstacles in the way to an
elementary limit formula for a Hurwitzian continued fraction with a
longer quasi-period: unless an almost miraculous simplification occurs
in an expression of Bessel functions of the ``wrong'' order, the sequence of
partial denominators must be complicated enough to be outside the class
studied in the present work. Finally, in Section~\ref{sec:examples2} we
provide an analogous description of all continued fractions in our
class for which the limit is expressed in terms of Bessel functions of
integer order. These Bessel functions are not known to be
elementary, but they are still widely studied. Calculating every third
convergent for a simple example in this class turns out to be
equivalent to calculating all convergents of a generalized
continued fraction. This ``coincidence'' seems worth a second look in
the future.  

Our work inspires several questions. Extending
the validity of our formulas to all convergents of the 
same class seems to require much more index management but only a few
more ideas. As indicated by Lehmer~\cite{DHLehmer}, the
calculation of the limits is easily extended to the class of
Hurwitzian fractions whose quasi-periodic part contains several
different constants and one arithmetic sequence. An explicit formula
for the convergents should be obtained using some multivariate
generalization of Fibonacci polynomials. The ideas used in our
calculations of the convergents may also be useful in finding
convergents of other Hurwitzian continued fractions. Finally, as indicated in
Sections~\ref{sec:lemma} and \ref{sec:conv}, there is combinatorics
behind the formulas for the convergents. This combinatorics would be
worth uncovering. A good starting point may be revisiting the
weighted lattice-path model proposed in Flajolet's work~\cite{Flajolet},
where convergents of generalized continued fractions arise as weights of
infinite lattice paths of limited height.

\section{Preliminaries}

\subsection{Continued fractions} A good reference on the
basic facts of the subject is Perron's classic work
\cite{Perron-v1,Perron-v2}. A {\em 
  generalized finite  
  continued fraction} is an expression of the form 
\begin{equation}
\label{eq:gcf}
a_0+\cfrac{b_1}{a_1+\cfrac{b_2}{a_2+\ddots\cfrac{b_{n-1}}{a_{n-1}+\cfrac{b_n}{a_n}}}},
\end{equation}
where the {\em initial term $a_0$}, the {\em partial denominators}
$a_1,\ldots,a_n$ and the {\em partial numerators} $b_1,\ldots,b_n$ may
be numbers or functions. An {\em infinite continued fraction} is 
obtained by letting $n$ go to infinity. The arising questions of
convergence have a reassuring answer for {\em continued
  fractions} where all $b_i$ equal $1$, the initial term $a_0$ is an
integer, and the partial denominators $a_i$ (for $i>0$) are positive
integers. Many sources refer to generalized continued fractions of the
form (\ref{eq:gcf}) as continued fractions and use the term
``simple continued fraction'' where we use ``continued fraction''.
All continued fractions in the present work are ``simple'' by default,
and we will use the distinguishing adjective ``generalized'' at the rare
occasion when it is needed. Every rational number may be written uniquely
as a finite continued fraction, subject to the restriction that the last
partial denominator is at least 
$2$ \cite[\S 9, Satz 2.1]{Perron-v1}. Every infinite 
continued fraction converges to an irrational number and every
irrational number may be uniquely written as a (necessarily infinite)
continued fraction \cite[\S 12, Satz 2.6]{Perron-v1}. As usual,
for a finite, respectively infinite, continued fraction we will use the
shorthand notations $[a_0,a_1,\ldots,a_n]$ and $[a_0,a_1,\ldots]$,
respectively. An infinite continued fraction
$[a_0,a_1,\ldots]$ is the limit of its {\em convergents}, that is, of
the finite continued fractions $[a_0,\ldots,a_n]$, obtained by
reading the first $n$ partial denominators. The convergents 
\begin{equation}
[a_0,\ldots,a_n]=\frac{p_n}{q_n}
\end{equation}
may be recursively computed from the initial conditions $p_{-1}=1$,
$p_0=a_0$, $q_{-1}=0$ and $q_0=1$, and from the recurrences
$p_n=a_n p_{n-1}+p_{n-2}$ and $q_n=a_nq_{n-1}+q_{n-2}$ for $n\geq 1$,
cf.\ \cite[\S2, Eq.\ (12), (13)]{Perron-v1}. The integers $p_n$ and
$q_n$ are relative prime for all $n$ \cite[\S 9, Satz
  2.1]{Perron-v1}. The {\em Euler-Mindig formulas},
derived for generalized continued fractions in \cite[\S 3]{Perron-v1}, allow
to express the numerators $p_n$ and the denominators $q_n$ directly. 
\begin{definition}
We call a set $S$ of integers {\em even} if it is the
disjoint union of intervals of even cardinality. 
Given two sets of integers $S$ and $T$ such that $S\subseteq T$, we 
say that {\em $T$ evenly contains $S$}, denoted by $S\subseteq_e T$ or
$T\supseteq_e S$, if $T\setminus S$ is an even set. 
\end{definition}
Thus, for example $\{5\}\subseteq_e \{1,\ldots,7\}$ since the difference
$\{1,2,3,4,6,7\}$ is the disjoint union of $\{1,\ldots,4\}$ and
$\{6,7\}$, whereas $\{4\}\not\subseteq_e \{1,\ldots, 4\}$ as $\{1,2,3\}$
can not be written as the disjoint union of intervals of even
cardinality. The Euler-Mindig formulas for continued fractions may be
restated as   
\begin{equation}
\label{eq:EM}
p_n=\sum_{S\subseteq_e \{0,\ldots,n\}} \prod_{i\in S} a_i 
\quad
\mbox{and}
\quad
q_n=\sum_{S\subseteq_e \{1,\ldots,n\}} \prod_{i\in S} a_i. 
\end{equation} 
The following observation is an immediate consequence of
Eq.\ (\ref{eq:EM}).
\begin{lemma}
\label{lem:shift}
The denominator $q_n$ associated to $[a_0,a_1,\ldots]$ is the same as
the numerator $p_{n-1}$ associated to $[a_1,a_2,\ldots]$.
\end{lemma}

\subsection{Hurwitzian continued fractions}
\label{ss:hurwitz}

A {\em Hurwitzian continued fraction} is a continued
fraction of the form
$[a_0,\ldots,a_h,\overline{\phi_0(\lambda),\ldots,\phi_{k-1}(\lambda)}]_{\lambda=0}^{\infty}$,
where $\phi_0, \ldots, \phi_{k-1}$ are polynomial functions that send
positive integers into positive integers. The
definition given in \cite[\S 32]{Perron-v1} is easily seen to be
equivalent. In the most trivial examples of a Hurwitzian continued
fraction all functions $\phi_j$ are constants, we then
obtain a {\em periodic continued fraction}. The set of real numbers
represented by a periodic continued fraction is exactly the set of
quadratic irrationals \cite[\S  20]{Perron-v1}. It should be noted that
quadratic irrationals are excluded from the definition of
Hurwitzian continued fractions in some recent
papers~\cite{Komatsu-longp}, by requiring that 
at least one of the repeatedly used polynomial functions be
non-constant. A real number is {\em Hurwitzian} if its 
continued fraction representation is a Hurwitzian continued
fraction. A famous Hurwitzian number is 
$$e=[2,1,2,1,1,4,1,1,6,1,1,8,1,1,10,\ldots],$$
see~\cite[\S34, Eq.\ (10)]{Perron-v1}. 
The main result on Hurwitzian numbers is Hurwitz theorem~\cite[\S
  33]{Perron-v1} stating that for a Hurwitzian number $\xi_0\in {\mathbb
  R}$, and any rational numbers $a,b,c,d$, satisfying $ad-bc\neq 0$, the
number $(a\xi_0+b)/(b\xi_0+d)$ is also a Hurwitzian number. Hurwitz theorem
provides also some estimate on the degrees of the polynomial functions
appearing in the continued fraction representation of
$(a\xi_0+b)/(b\xi_0+d)$, remains silent however on the issue how the
length of the period may be affected by the fractional linear 
transformation $\xi_0\mapsto (a\xi_0+b)/(b\xi_0+d)$. A generalization of
Hurwitz theorem may be found in~\cite{Stambul}.   

Hurwitzian numbers may be computed from their continued fraction
representation in some special cases, the expansion of the set of
examples is subject of ongoing research. A frequently overlooked first
attempt may be found in Perron's book~\cite[\S 48,
  Satz 6.3]{Perron-v2} which states (without proof) a formula for
generalized continued fractions of the form 
(\ref{eq:gcf}) having the property that the partial numerators $b_i$ are
all equal, and that the numbers $a_i$ form an arithmetic sequence. For
continued fractions, Perron's result gives
\begin{equation}
\label{eq:perron}
[\beta_0,\beta_0+\beta_1,\beta_0+2\beta_1,\beta_0+3\beta_1,\ldots]=\beta_1\frac{\sum\limits_{n=0}^{\infty} \frac{1}{\beta_1^{2n} n!
  \Gamma\left(\frac{\beta_0}{\beta_1}+n\right)}}
  {\sum\limits_{n=0}^{\infty} \frac{1}{\beta_1^{2n} n! 
  \Gamma\left(\frac{\beta_0}{\beta_1}+n+1\right)}}. 
\end{equation}
The same class of Hurwitzian continued fractions was revisited by
D.\ H.\ Lehmer \cite{DHLehmer}, who proved the following formula. 
\begin{equation}
\label{eq:lehmer}
[\beta_0,\beta_0+\beta_1,\beta_0+2\beta_1,\beta_0+3\beta_1,\ldots]
=I_{\beta_0/\beta_1-1}(2\beta_1^{-1})/I_{\beta_0/\beta_1}(2\beta_1^{-1})
\end{equation}
where $I_{\nu}(z)$ is the modified Bessel function of the first kind
\begin{equation}
\label{eq:mBessel}
I_{\nu}(z)
=\sum_{m=0}^{\infty} \frac{(z/2)^{\nu+2m}}{\Gamma(m+1)\Gamma(\nu+m+1)}. 
\end{equation}
D.\ H.\ Lehmer's result was generalized in recent papers to other Hurwitzian
continued fractions
\cite{Komatsu-HT,Komatsu-longp,McLaughlin}. Some analogous results were found
for {\em Tasoev continued fractions}
\cite{Komatsu-HT,Komatsu-RR,McLaughlin,McLaughlin-Wyshinski}, defined as
continued fractions of the form  
$[a_0;\overline{\underbrace{a^k,\ldots,a^k}_{m}}]_{k=1}^{\infty}$. 
As pointed out by T.\ Komatsu~\cite{Komatsu-longp}, most of these recent
results involve Hurwitzian and Tasoev continued fractions where the
length of the quasi-period does not exceed $3$, Komatsu's
work~\cite{Komatsu-longp} contains some sophisticated examples of longer
quasi-periods. 

\subsection{Bessel functions}
\label{sec:Bessel}

Besides the modified Bessel functions $I_{\nu}(x)$, our formulas
will also involve the (original) Bessel functions of the first kind 
\begin{equation}
\label{eq:Bessel}
J_{\nu}(z)
=\sum_{m=0}^{\infty} \frac{(-1)^m (z/2)^{\nu+2m}}{\Gamma(m+1)\Gamma(\nu+m+1)}. 
\end{equation}
There is an elementary expression for the functions $I_{\nu}(z)$ and
$J_{\nu}(z)$ respectively, whenever $\nu$ is the half of an odd
integer. Indeed, using $\Gamma(1/2)=\sqrt{\pi}$ and
$\Gamma(z+1)=z\Gamma(z)$, it is easy to derive directly from
the definitions (\ref{eq:mBessel}) and (\ref{eq:Bessel}) that we have  
\begin{equation}
\label{eq:I1/2}
I_{-1/2}(z)=\sqrt{\frac{2}{\pi z}}\cosh(z),\quad 
I_{1/2}(z)=\sqrt{\frac{2}{\pi z}}\sinh (z), 
\end{equation}
\begin{equation}
\label{eq:J1/2}
J_{-1/2}(z)=\sqrt{\frac{2}{\pi z}}\cos(z), 
\quad\mbox{and}\quad
J_{1/2}(z)=\sqrt{\frac{2}{\pi z}}\sin (z). 
\end{equation}
(see \cite[List of formul\ae: 44, 48, 182,
  186]{McLachlan}). The functions 
$I_{\nu}(z)$ and $J_{\nu}(z)$ satisfy very similar recurrence formulas:
$$
I_{\nu+1}(x)=I_{\nu-1}(x)-\frac{2\nu}{x} I_{\nu}(x) 
\quad\mbox{and}\quad
J_{\nu+1}(x)=-J_{\nu-1}(x)+\frac{2\nu}{x} J_{\nu}(x), 
$$
see \cite[9.1.27, 9.6.26]{Abramowitz-Stegun}. These allow us to find explicit
elementary expressions for  $I_{\nu}(z)$ and $J_{\nu}(z)$, whenever
$\nu$ is the half of an odd integer. In particular, for $\nu=3/2$, we obtain
\begin{equation}
\label{eq:IJ3/2}
I_{3/2}(x)=\sqrt{\frac{2}{\pi x}}\left(\cosh(x)-\frac{\sinh
  (x)}{x}\right)
\quad\mbox{and}\quad
J_{3/2}(x)
=\sqrt{\frac{2}{\pi x}}\left(\frac{\sin(x)}{x}-\cos(x)\right),
\end{equation}
see \cite[List of formul\ae: 45, 183]{McLachlan}.

\subsection{Fibonacci and Lucas polynomials}

The Fibonacci and Lucas polynomials are $q$-analogues of the usual
Fibonacci and Lucas numbers. 

\begin{definition}
\label{def:qfib}
We define the {\em Fibonacci polynomials $F_n(q)$} and the {\em Lucas
  polynomials $L_n(q)$} by the initial conditions $F_0(q)=0$,
$F_1(q)=1$, $L_0(q)=2$ and $L_1(q)=q$; and by the common recurrence 
$X_n(q)=q\cdot X_{n-1}(q)+X_{n-2}(q)$ for $n\geq 2$, where the letter
$X$ should be replaced by either $F$ or $L$ throughout the defining recurrence.
\end{definition}
Fibonacci and Lucas polynomials are widely studied, they even have their
own Mathworld and Wikipedia entries. As a sample reference on Fibonacci
polynomials see Yuang and Zhang~\cite{Yuang-Zhang}, for
further generalizations, see Cigler~\cite{Cigler}.

The coefficients of the Fibonacci and Lucas polynomials 
are listed as sequences A102426 and A034807 in the Online
Encyclopedia of Integer Sequences~\cite{OEIS}. In a recent work of Foata and
Han~\cite{Foata-Han} the  coefficient of $q^{n+1-m}$ in
$F_{n+1}(q)$ appears as the number of {\em   $t$-compositions} of $n$ into $m$
parts. Here we record another combinatorial interpretation. The
coefficient of $q^{n+1-m}$ in $F_{n+1}(q)$ is the number of $m$-element
even sets contained in $\{1,\ldots,n\}$:   
\begin{equation}
\label{eq:Fq}
F_{n+1}(q)=\sum_{\emptyset\subseteq_e S\subseteq \{1,\ldots,n\}} q^{n-|S|}
=\sum_{S\subseteq_e \{1,\ldots,n\}} q^{|S|}.
\end{equation}
Eq.\ (\ref{eq:Fq}) may be shown by induction on $n$, using the
defining recurrence. Just like for the usual Fibonacci and Lucas
numbers, a closed formula for $F_n(q)$ and $L_n(q)$ may be obtained after
solving the characteristic equation. We have
\begin{eqnarray}
\label{eq:Fnqc}
F_n(q)&=&\frac{1}{\sqrt{q^2+4}}\left(\rho_1^{n}-\rho_2^{n}\right)\quad\mbox{and}\\
\label{eq:Lq}
L_n(q)&=&\rho_1^{n}+\rho_2^{n}
\end{eqnarray}
where 
\begin{equation}
\label{eq:rhos}
\rho_1=\frac{q+\sqrt{q^2+4}}{2}\quad\mbox{and}\quad
\rho_2=\frac{q-\sqrt{q^2+4}}{2}.
\end{equation}

\section{Our main result}
\label{sec:main}

All our results will be about Hurwitzian continued fractions of the form    
$$
\xi(\alpha,\beta_0,\beta_1,d,r):=
[\underbrace{\alpha,\ldots,\alpha}_r,\beta_0,
\overline{\underbrace{\alpha,\ldots,\alpha,\beta_0+\beta_1\cdot
    n}_d}]_{n=1}^\infty
$$
where $\alpha$, $\beta_0$, $\beta_1$ and $d$ are positive integers and 
and $r$ is a nonnegative integer. 
In order to state them we need to introduce two {\em magic numbers}
associated to such a continued fraction.
\begin{definition}
\label{def:magic}
The {\em magic sum} associated to $\xi(\alpha,\beta_0,\beta_1,d,r)$
is the sum
$$
\sigma(\alpha,\beta_0,\beta_1,d,r)
=
\frac{\beta_0-\alpha}{\beta_1}+\frac{L_d(\alpha)}{\beta_1 F_{d}(\alpha)}, 
$$
the {\em magic quotient} associated to
$\xi(\alpha,\beta_0,\beta_1,d,r)$
is the quotient
$$
\rho(\alpha,\beta_0,\beta_1,d,r)=\frac{(-1)^{d-1} 
    }{\beta_1^2\cdot F_{d}(\alpha)^2}.
$$
\end{definition}
Whenever this does not lead to confusion, we will omit the parameters 
$(\alpha,\beta_0,\beta_1,d,r)$ and denote the magic numbers simply by
$\sigma$ and $\rho$, respectively. Note that $\sigma$ does not depend on
$r$ and that $\rho$ depends only on $\alpha$, $\beta_1$ and $d$. Our
main result is the following.
\begin{theorem}
\label{thm:hurwitzc}
The $(nd+r-1)$st convergent of $\xi(\alpha,\beta_0,\beta_1,d,r)$ may be
written as $p_{nd+r-1}/q_{nd+r-1}$ where $p_{nd+r-1}$ is given by 
$$
\begin{array}{rcl}
\di \frac{p_{nd+r-1}}{F_{d}(\alpha)^n\beta_1^n}
&=&
\di F_{r+1}(\alpha)
\sum_{k=0}^{\lfloor n/2\rfloor}
\frac{(n-k)!}{k!}\binom{n+\sigma-1-k}{n-2k} \rho^k
+(-1)^{d-r}F_{d-r-1}(\alpha)\cdot 
\\
&&\di F_{d}(\alpha)\beta_1 
\sum_{k=0}^{\lfloor (n-1)/2 \rfloor}
\frac{(n-k-1)!}{k!}\binom{n+\sigma-1-k}{n-2k-1}\rho^{k+1},\\
\end{array}
$$
and $q_{nd+r-1}$ is given by
$$
\begin{array}{rcl}
\di \frac{q_{nd+r-1}}{F_{d}(\alpha)^n\beta_1^n}
&=&
\di F_{r}(\alpha)
\sum_{k=0}^{\lfloor n/2\rfloor}
\frac{(n-k)!}{k!}\binom{n+\sigma-1-k}{n-2k} \rho^k
+(-1)^{d+1-r}F_{d-r}(\alpha)\cdot 
\\
&&\di F_{d}(\alpha)\beta_1 
\sum_{k=0}^{\lfloor (n-1)/2 \rfloor}
\frac{(n-k-1)!}{k!}\binom{n+\sigma-1-k}{n-2k-1}\rho^{k+1}.\\
\end{array}
$$
\end{theorem}
We postpone the proof of Theorem~\ref{thm:hurwitzc} till
Section~\ref{sec:conv}. In this section we only show how to find
an asymptotic formula for the numerator and the denominator, which
yields essentially the same limit formula that can also be derived from
Lehmer's work~\cite{DHLehmer}. This is the only part of our paper where
the notion of limits and convergence from analysis will be used, the
proof of Theorem~\ref{thm:hurwitzc} will only involve purely algebraic
and combinatorial manipulations. We begin with estimating the magic sum
$\sigma$.

\begin{lemma}
The magic sum $\sigma$ is always positive.
\end{lemma}
\begin{proof}
Since $\beta_0$ and $\beta_1$ are positive, it suffices to prove that
$L_d(\alpha) \geq \alpha F_{d}(\alpha)$ holds for all positive integer
$\alpha$ and all nonnegative integer $d$. This may be shown by induction
on $d$, using the common recurrence of the Fibonacci  and Lucas polynomials.
\end{proof}
\begin{corollary}
\label{cor:spos}
The falling factorial 
$(\sigma+n-1)_n=(\sigma+n-1)\cdots (\sigma)$ is always positive. 
\end{corollary}
As a consequence, we may divide $p_{nd+r-1}/(F_{d}(\alpha)^n\beta_1^n)$,
as well as $q_{nd+r-1}/(F_{d}(\alpha)^n\beta_1^n)$ by $(\sigma+n-1)_n$ and
ask whether these quotients converge as $n$ goes to infinity. It will
turn out that they do.

\begin{lemma}
\label{lem:rlimit}
Let $\sigma$ be any positive real number and $\rho$ be any real
number. Then
$$
\lim_{n\rightarrow \infty} \frac{1}{(\sigma+n-1)_n} 
\sum_{k=0}^{\lfloor n/2\rfloor}
\frac{(n-k)!}{k!}\binom{n+\sigma-1-k}{n-2k} \rho^k
=
\sum_{m=0}^{\infty} \frac{\rho^m}{m!(\sigma+m-1)_m}.
$$ 
\end{lemma}
\begin{proof}
Note first that the series on the right hand side is absolute convergent
for any $\rho$ as each of its term has smaller absolute value than the
corresponding term of $e^{\rho/\sigma}=\sum_{m=0}^{\infty} \rho^m
\sigma^{-m}/m!$, an absolute convergent series. For any $\varepsilon >0$
there is an $M$ such that 
\begin{equation}
\label{eq:e1}
\sum_{m=M+1}^{\infty} \left|\frac{\rho^m}{m!(\sigma+m-1)_m}\right|
<\frac{\varepsilon}{3}. 
\end{equation}
Let us consider now the expression of $n$ on the left hand side. It may
be rewritten as 
$$
\sum_{k=0}^{\lfloor n/2\rfloor} \frac{\rho^k}{k!(\sigma+k-1)_k} \cdot
\delta_{n,k} \quad\mbox{where}\quad
\delta_{n,k}=\frac{(n-k)_k}{(\sigma+n-1)_k}.$$
The factor $\delta_{n,k}$ may be estimated as follows:
$$
\left(\frac{n-2k+1}{\sigma+n-k}\right)^k\leq 
\delta_{n,k}=\frac{(n-k)\cdots (n-2k+1)}{(\sigma+n-1)\cdots (\sigma+n-k)}
\leq \left(\frac{n-k}{\sigma+n-1}\right)^k.
$$
Our upper bound for $\delta_{n,k}$ is $1$ for $k=0$ and it is less than
$1$ for $k>0$, since $n+\sigma-1>n-1>n-k$. For fixed $k$, our lower
bound converges to $1$ as $n$ goes to infinity. Now, for all $n>M$ we
may write 
$$
\sum_{k=0}^{\lfloor n/2\rfloor} \frac{\rho^k}{k!(\sigma+k-1)_k} \cdot
\delta_{n,k}
=
\sum_{k=0}^{M} \frac{\rho^k}{k!(\sigma+k-1)_k} \cdot
\delta_{n,k}+
\sum_{k=M+1}^{\lfloor n/2\rfloor} \frac{\rho^k}{k!(\sigma+k-1)_k} \cdot
\delta_{n,k}
$$
By our choice of $M$, and by $|\delta_{n,k}|\leq 1$, the second sum on
the right hand side has absolute value less than $\varepsilon/3$, for
all $n>M$. As $n$ goes to infinity, each of $\delta_{n,0}$, \ldots,
$\delta_{n,M}$ converges to $1$, so there is an $n_0\geq M$ such that for all
$n>n_0$ we have 
$$
\left|\sum_{k=0}^{M} \frac{\rho^k}{k!(\sigma+k-1)_k} \cdot
\delta_{n,k}- \sum_{k=0}^{M} \frac{\rho^k}{k!(\sigma+k-1)_k}
\right|<\frac{\varepsilon}{3}, 
$$  
implying
\begin{equation}
\label{eq:e2}
\left|\sum_{k=0}^{\lfloor n/2\rfloor} \frac{\rho^k}{k!(\sigma+k-1)_k} \cdot
\delta_{n,k}- \sum_{k=0}^{M} \frac{\rho^k}{k!(\sigma+k-1)_k}
\right|<\frac{2 \varepsilon}{3}, 
\end{equation}
Combining (\ref{eq:e1}) and (\ref{eq:e2}), we obtain
$$
\left|\sum_{k=0}^{\lfloor n/2\rfloor} \frac{\rho^k}{k!(\sigma+k-1)_k} \cdot
\delta_{n,k}-\sum_{k=0}^{\infty} \frac{\rho^k}{k!(\sigma+k-1)_k}\right|
< \varepsilon \quad\mbox{for all $n\geq n_0$.}
$$
\end{proof}
\begin{lemma}
\label{lem:slimit}
Let $\sigma$ be any positive real number and $\rho$ be any real
number. Then
$$
\lim_{n\rightarrow \infty} \frac{1}{(\sigma+n-1)_n} 
\sum_{k=0}^{\lfloor (n-1)/2\rfloor}
\frac{(n-k-1)!}{k!}\binom{n+\sigma-1-k}{n-2k-1} \rho^{k+1}
=
\sum_{m=0}^{\infty} \frac{\rho^{m+1}}{m!(\sigma+m)_{m+1}}
$$ 
\end{lemma}
\begin{proof}
The proof is very similar to the proof of
Lemma~\ref{lem:rlimit}, thus we omit the details. We only note that this
time we have 
$$
\frac{1}{(\sigma+n-1)_n} 
\sum_{k=0}^{\lfloor (n-1)/2\rfloor}
\frac{(n-k-1)!}{k!}\binom{n+\sigma-1-k}{n-2k-1} \rho^{k+1}
=
\sum_{k=0}^{\lfloor (n-1)/2\rfloor}
\frac{\rho^{k+1}}{k! (\sigma+k)_{k+1}} \cdot \delta'_{n,k}
$$
where
$$
\delta'_{n,k}=\frac{(n-k-1)_k}{(\sigma+n-1)_k}.
$$
Again $|\delta'_{n,k}|\leq 1$ for all $n$ and $k$ and, for any fixed
$k$, the limit of $\delta'_{n,k}$ is $1$ as $n$ goes to infinity.
\end{proof}
Combining Lemmas~\ref{lem:rlimit} and \ref{lem:slimit} with
Theorem~\ref{thm:hurwitzc}, we obtain 
\begin{equation}
\label{eq:plim}
\begin{array}{rcl}
\di \lim_{n\rightarrow\infty} \frac{p_{nd+r-1}}{F_{d}(\alpha)^n\beta_1^n
  (\sigma+n-1)_n} 
&=&
\di F_{r+1}(\alpha)
\sum_{m=0}^{\infty} \frac{\rho^m}{m!(\sigma+m-1)_m}\\
&&\di +(-1)^{d-r}F_{d-r-1}(\alpha) F_d(\alpha)\beta_1
\sum_{m=0}^{\infty} \frac{\rho^{m+1}}{m!(\sigma+m)_{m+1}}\\
\end{array}
\end{equation}
and 
\begin{equation}
\label{eq:qlim}
\begin{array}{rcl}
\di \lim_{n\rightarrow\infty} \frac{q_{nd+r-1}}{F_{d}(\alpha)^n\beta_1^n
  (\sigma+n-1)_n} 
&=&
\di F_{r}(\alpha)
\sum_{m=0}^{\infty} \frac{\rho^m}{m!(\sigma+m-1)_m}\\
&&\di +(-1)^{d-r+1}F_{d-r}(\alpha) F_d(\alpha)\beta_1
\sum_{m=0}^{\infty} \frac{\rho^{m+1}}{m!(\sigma+m)_{m+1}}.\\
\end{array}
\end{equation}
Taking the quotients of the right hand sides of (\ref{eq:plim}) and
(\ref{eq:qlim}), we obtain the following formula for
$\xi=\xi(\alpha,\beta_0,\beta_1,d,r)$: 
\begin{equation}
\label{eq:thetruth}
\xi
=
\frac{\di F_{r+1}(\alpha)
\sum_{m=0}^{\infty} \frac{\rho^m}{m!(\sigma+m-1)_m}
+(-1)^{d-r}F_{d-r-1}(\alpha) F_d(\alpha)\beta_1
\sum_{m=0}^{\infty} \frac{\rho^{m+1}}{m!(\sigma+m)_{m+1}}}
{\di F_{r}(\alpha)
\sum_{m=0}^{\infty} \frac{\rho^m}{m!(\sigma+m-1)_m}
+(-1)^{d-r+1}F_{d-r}(\alpha) F_d(\alpha)\beta_1
\sum_{m=0}^{\infty} \frac{\rho^{m+1}}{m!(\sigma+m)_{m+1}}}.
\end{equation}
This is the formula that could have been discovered in a ``parallel
universe'' where interest in Hurwitzian continued fractions were to
arise a long time before defining Bessel functions.   
We may restate Eq.\ (\ref{eq:thetruth}) in terms of Bessel functions
using the following two obvious statements.
\begin{lemma}
\label{lem:mBessel}
For any positive $\rho$ and $\sigma$ we have 
$$
\sum_{m=0}^{\infty} \frac{\rho^m}{m!(\sigma+m-1)_m}
=\frac{I_{\sigma-1}(2\sqrt{\rho})\Gamma(\sigma)}
{\rho^{(\sigma-1)/2}}
\quad
\mbox{and}
\quad
\sum_{m=0}^{\infty} \frac{\rho^{m+1}}{m!(\sigma+m)_{m+1}}
=\frac{I_{\sigma}(2\sqrt{\rho})\Gamma(\sigma)}{\rho^{(\sigma-2)/2}}.
$$
Here $I_{\nu}(z)$ is the modified Bessel function defined in
(\ref{eq:mBessel}).  
\end{lemma}
\begin{lemma}
\label{lem:Bessel}
For any negative $\rho$ and positive $\sigma$ we have 
$$
\sum_{m=0}^{\infty} \frac{\rho^m}{m!(\sigma+m-1)_m}
=\frac{J_{\sigma-1}(2\sqrt{-\rho})\Gamma(\sigma)}
{(-\rho)^{(\sigma-1)/2}}
\quad
\mbox{and}
$$
$$
\sum_{m=0}^{\infty} \frac{\rho^{m+1}}{m!(\sigma+m)_{m+1}}
=-\frac{J_{\sigma}(2\sqrt{-\rho})\Gamma(\sigma)}{(-\rho)^{(\sigma-2)/2}}. 
$$
Here $J_{\nu}(z)$ is the Bessel function given by (\ref{eq:Bessel}).
\end{lemma}
Using Lemmas~\ref{lem:Bessel} and \ref{lem:mBessel} above we may
rephrase Eq.\ (\ref{eq:thetruth}) as follows. 
\begin{theorem}
\label{thm:hurwitzn}
Let $\alpha$, $\beta_0$, $\beta_1$  and $d$ be positive integers, and
let $r$ be nonnegative integers. Then the Hurwitzian continued
fraction   
$$
\xi(\alpha,\beta_0,\beta_1,d,r)=
[\underbrace{\alpha,\ldots,\alpha}_r,\beta_0,
\overline{\underbrace{\alpha,\ldots,\alpha,\beta_0+\beta_1\cdot
    n}_d}]_{n=1}^\infty
$$
is given by 
$$
\xi(\alpha,\beta_0,\beta_1,d,r)=
\frac{\di F_{r+1}(\alpha)
I_{\sigma-1}(2\sqrt{\rho})
+(-1)^{r+1} F_{d-r-1}(\alpha) I_{\sigma}(2\sqrt{\rho})}
{\di F_{r}(\alpha) I_{\sigma-1}(2\sqrt{\rho})
+(-1)^{r} F_{d-r}(\alpha) I_{\sigma}(2\sqrt{\rho})}
$$
if $d$ is odd, and it is given by 
$$
\xi(\alpha,\beta_0,\beta_1,d,r)
=
\frac{\di F_{r+1}(\alpha)
J_{\sigma-1}(2\sqrt{-\rho})
+(-1)^{r+1}F_{d-r-1}(\alpha) 
J_{\sigma}(2\sqrt{-\rho})}
{\di F_{r}(\alpha)
J_{\sigma-1}(2\sqrt{-\rho})
+(-1)^{r}F_{d-r}(\alpha) 
J_{\sigma}(2\sqrt{-\rho})},
$$
if $d$ is even. Here $I_{\nu}(z)$ and  $J_{\nu}(z)$, respectively, denotes the 
is the modified respectively original Bessel function  defined in
(\ref{eq:mBessel}) and (\ref{eq:Bessel}), respectively. 
\end{theorem}
\begin{proof}
We work out only the case of odd $d$ in detail, the case of even $d$ is
completely analogous. Direct substitution of Lemma~\ref{lem:mBessel} into
(\ref{eq:thetruth}) yields
$$
\xi(\alpha,\beta_0,\beta_1,d,r)
=
\frac{\di F_{r+1}(\alpha)
\frac{I_{\sigma-1}(2\sqrt{\rho})}
{\rho^{(\sigma-1)/2}}
+(-1)^{d-r}F_{d-r-1}(\alpha) F_d(\alpha)\beta_1
\frac{I_{\sigma}(2\sqrt{\rho})}{\rho^{(\sigma-2)/2}}}
{\di F_{r}(\alpha)
\frac{I_{\sigma-1}(2\sqrt{\rho})}
{\rho^{(\sigma-1)/2}}
+(-1)^{d-r+1}F_{d-r}(\alpha) F_d(\alpha)\beta_1
\frac{I_{\sigma}(2\sqrt{\rho})}{\rho^{(\sigma-2)/2}}}.
$$
After multiplying the numerator and the denominator by
$\rho^{(\sigma-1)/2}$ we get
$$
\xi(\alpha,\beta_0,\beta_1,d,r)=
\frac{\di F_{r+1}(\alpha)
I_{\sigma-1}(2\sqrt{\rho})
+(-1)^{d-r} \sqrt{\rho} F_{d-r-1}(\alpha) F_d(\alpha)\beta_1
I_{\sigma}(2\sqrt{\rho})}
{\di F_{r}(\alpha) I_{\sigma-1}(2\sqrt{\rho})
+(-1)^{d-r+1} \sqrt{\rho} F_{d-r}(\alpha) F_d(\alpha)\beta_1 
I_{\sigma}(2\sqrt{\rho})}.
$$
The first equality in Theorem~\ref{thm:hurwitzn} follows from
$\sqrt{\rho}=\beta_1^{-1} 
F_d(\alpha)^{-1}$ and from the fact that $(-1)^d=(-1)$ in this
case. A similar reasoning for even $d$ yields
$$
\xi(\alpha,\beta_0,\beta_1,d,r)
=
\frac{\di F_{r+1}(\alpha)
J_{\sigma-1}(2\sqrt{-\rho})
-(-1)^{d-r}F_{d-r-1}(\alpha) 
J_{\sigma}(2\sqrt{-\rho})}
{\di F_{r}(\alpha)
J_{\sigma-1}(2\sqrt{-\rho})
-(-1)^{d-r+1}F_{d-r}(\alpha) 
J_{\sigma}(2\sqrt{-\rho})},
$$
a final simplification may be made by observing that $(-1)^d=1$ in this case.
\end{proof}
\begin{remark}
\label{rem:Lehmersway}
Theorem~\ref{thm:hurwitzn} may also be derived from 
Lehmer's work~\cite{DHLehmer} directly, using a few, easily verifiable facts
about Fibonacci and Lucas polynomials. Again, we outline the proof for odd $d$
only, the case of even $d$ being completely analogous. Consider first the case
when $d=1$ and $r=0$. In this case we get $\sigma=\beta_0/\beta_1$,
regardless of of $\alpha$. Theorem~\ref{thm:hurwitzn} takes the form 
$\xi(\alpha,\beta_0,\beta_1,1,0)=
I_{\sigma-1}(2\sqrt{\rho})/I_{\sigma}(2\sqrt{\rho})$,  
which, by $\rho=\beta_1^{-2}$, is exactly Lehmer's formula
(\ref{eq:lehmer}). Consider next the case when $d$ is an arbitrary odd
integer and $r=0$. In this case, we may use Lehmer's~\cite[Theorem 4]{DHLehmer}
with all constants being equal to $\alpha$. Using the fact that 
$$
[\underbrace{\alpha,\ldots,\alpha}_d]=\frac{F_{d+1}(\alpha)}{F_d(\alpha)},
$$
the fractions $A/B$, $A'/B'$ and $A''/B''$ appearing in~\cite[Theorem
  4]{DHLehmer} are easily seen to correspond to
$((\beta_0-\alpha)F_d(\alpha)+F_{d+1}(\alpha))/F_d(\alpha)$,
$((\beta_0-\alpha)F_{d-1}(\alpha)+F_{d}(\alpha))/F_{d-1}(\alpha)$ and
$F_{d}(\alpha)/F_{d-1}(\alpha)$, respectively, in our
notation. Lehmer's $(bB+B'+B'')/(aB)$ corresponds to 
our 
$$
\frac{\beta_0 F_d(\alpha)+2F_{d-1}(\alpha)}{\beta_1F_d(\alpha)}
=
\frac{(\beta_0-\alpha)F_d(\alpha)+(\alpha F_d(\alpha)+
  2F_{d-1}(\alpha))}{\beta_1F_d(\alpha)} 
=
\sigma,
$$
since $\alpha F_d(\alpha)+
  2F_{d-1}(\alpha)=L_d(\alpha)$ holds for all $d$. 
Lehmer's~\cite[Theorem 4]{DHLehmer} gives 
$$
\xi(\alpha,\beta_0,\beta_1,d,0)
=\frac{1}{F_d(\alpha)}
\left(-F_{d-1}(\alpha)
+\frac{I_{\sigma-1}(2(\beta_1F_{d}(\alpha))^{-1})}
  {I_{\sigma}(2(\beta_1F_{d}F_{d}(\alpha))^{-1})}\right) 
$$
which is the same as the formula implied by Theorem~\ref{thm:hurwitzn},
after noting that we have $\rho=(\beta_1F_{d}(\alpha))^{-2}$. Finally,
for arbitrary $r$ we may substitute
$\eta:=\xi(\alpha,\beta_0,\beta_1,d,0)$ into the formula
$$
[\underbrace{\alpha,\ldots,\alpha}_r, \eta]=
\frac{F_{r+1}(\alpha)\eta+F_r(\alpha)} 
{F_{r}(\alpha)\eta+F_{r-1}(\alpha)} 
$$
and obtain
$$
\xi(\alpha,\beta_0,\beta_1,d,r)=
\frac{F_{r+1}(\alpha)\left(\frac{1}{F_d(\alpha)}
\left(-F_{d-1}(\alpha)
+\frac{I_{\sigma-1}(2(\beta_1F_{d}(\alpha))^{-1})}
  {I_{\sigma}(2(\beta_1F_{d}F_{d}(\alpha))^{-1})}\right)+F_r(\alpha)\right)}
{F_{r}(\alpha)\left(\frac{1}{F_d(\alpha)}
\left(-F_{d-1}(\alpha)
+\frac{I_{\sigma-1}(2(\beta_1F_{d}(\alpha))^{-1})}
  {I_{\sigma}(2(\beta_1F_{d}F_{d}(\alpha))^{-1})}\right)+F_{r-1}(\alpha)\right)}.
$$
(Recall that $\sigma$ does not depend on $r$.) Equivalently, 
$$
\xi=
\frac{F_{r+1}(\alpha)I_{\sigma-1}(2(\beta_1F_{d}(\alpha))^{-1})
+\left(F_r(\alpha)F_d(\alpha)-F_{r+1}(\alpha)F_{d-1}(\alpha)\right)I_{\sigma}(2(\beta_1F_{d}F_{d}(\alpha))^{-1})}
{F_{r}(\alpha) I_{\sigma-1}(2(\beta_1F_{d}(\alpha))^{-1})+
  \left(F_{r-1}(\alpha)F_d(\alpha)-F_{r}(\alpha)F_{d-1}(\alpha)\right)I_{\sigma}(2(\beta_1F_{d}F_{d}(\alpha))^{-1})}.
$$
Theorem~\ref{thm:hurwitzn} now follows after observing that
$F_r(\alpha)F_d(\alpha)-F_{r+1}(\alpha)F_{d-1}(\alpha)$ and and
$F_{r-1}(\alpha)F_d(\alpha)-F_{r}(\alpha)F_{d-1}(\alpha)$, respectively,
may be replaced by $(-1)^{r+1}F_{d-r-1}(\alpha)$  and $(-1)^r
F_{d-r}(\alpha)$  respectively.
\end{remark}

\section{Two useful lemmas}
\label{sec:lemma}

In this section we provide a combinatorial proof for two polynomial
identities which seem to be interesting by their own right. They will
play a crucial role in Section~\ref{sec:conv} where we prove
Theorem~\ref{thm:hurwitzc}. Our lemmas will be summation formulas for
the polynomials
\begin{equation}
\label{eq:rxy}
R_n(x,y):=\frac{1}{n!}\sum_{k=0}^n \binom{n}{k} \binom{n+y}{n-k} (n-k)!^2 x^k
\end{equation}
and
\begin{equation}
\label{eq:sxy}
S_n (x,y):=
\frac{1}{n!}\sum_{k=0}^{n-1} \binom{n}{k} \binom{n+y}{n-k-1}
(n-k)!(n-k-1)! x^{k+1} 
\end{equation}
\begin{lemma}
\label{lem:rsum}
The polynomials $R_n(x,y)$ satisfy 
$$
\sum_{m=0}^n \frac{(-x)^{n-m}}{(n-m)!} R_m(x,y)
=\frac{1}{n!}\sum_{k=0}^{\lfloor n/2\rfloor}
\binom{n}{k}\binom{n+y-k}{n-2k}(n-k)!^2 x^k 
\quad\mbox{for all $n\geq 0$}.
$$
\end{lemma}
\begin{proof}
We consider both sides of the equation as polynomials in the variable
$y$ with coefficients from the field ${\mathbb Q}(x)$. Since a nonzero
polynomial has only finitely many roots, it suffices to show that the
two sides equal for any nonnegative integer value of $y$. 

For $y\in{\mathbb N}$, the left hand side may then be rewritten as 
$$
\frac{1}{n!}\sum_{m=0}^n \binom{n}{n-m} (-x)^{n-m} 
\sum_{k=0}^m \binom{m}{k} \binom{m+y}{y+k} (m-k)!^2 x^k.
$$
This is $1/n!$ times the total weight of all quadruplets 
$(\pi_1,\pi_2,\gamma_1,\gamma_2)$ subject to the
following conditions:
\begin{enumerate}
\item $\pi_1$ is a permutation of $\{1,\ldots,n\}$, $\pi_2$ is a
  permutation of $\{1,\ldots,n+y\}$, the functions $\gamma_1:
  \{1,\ldots,n\}\rightarrow \{0,1,2\}$ and  $\gamma_2:
  \{1,\ldots,n+y\}\rightarrow \{0,1,2\}$ are colorings;
\item\label{it:12} for $i\in\{1,2\}$ an element $j$ satisfying
  $\gamma_i(j)=1$ must be a fixed point of $\pi_i$; 
\item\label{it:c} for any $j$, $\gamma_1(j)=2$ is equivalent to
  $\gamma_2(j)=2$ and $\gamma_1(j)=\gamma_2(j)=2$ implies that $j$ is a
  common fixed point  of $\pi_1$ and $\pi_2$;
\item\label{it:exc} the colorings $\gamma_1$ and $\gamma_2$ satisfy $|\{j\in
  \{1,\ldots,n+y\}\::\: \gamma_2(j)>0\}|=|\{j\in \{1,\ldots,n\}\::\:
  \gamma_1(j)>0\}|+y$;   
\end{enumerate}
We define the weight of $(\pi_1,\pi_2,\gamma_1,\gamma_2)$ as
$$x^{|\{j\in \{1,\ldots,n\}\::\: \gamma_1(j)=1\}|}\cdot (-x)^{|\{j\in
  \{1,\ldots,n\}\::\: \gamma_1(j)=2\}|}.$$ 
Indeed, after setting $n-m$ as the number of elements $j$ satisfying
$\gamma_1(j)=\gamma_2(j)=2$, there are $\binom{n}{n-m}$ ways to select
them. By rule (\ref{it:c}) these are common fixed points of $\pi_1$
and $\pi_2$ (and thus elements of $\{1,\ldots,n\}$). Next we set $k$
as the number of  elements $j\in \{1,\ldots,n\}$ satisfying
$\gamma_1(j)=1$, and select these elements, in $\binom{m}{k}$ ways. 
The remaining elements of $\{1,\ldots,n\}$ satisfy $\gamma_1(j)=0$.
At this point we have $|\{j\in \{1,\ldots,n\}\::\:
  \gamma_1(j)>0\}|=n-m+k$. Thus, by rule (\ref{it:exc}), the set
$\{j\in \{1,\ldots,n+y\}\::\: \gamma_2(j)>0\}$ must have $n-m+k+y$
  elements. Exactly $n-m$ of these elements satisfy $\gamma_2(j)=2$,
  therefore the number of elements $j\in \{1,\ldots,n+y\}$ satisfying
$\gamma_2(j)=1$ must be $y+k$. There are $\binom{m+y}{y+k}$ ways to
  select them. So far we have selected $n-m+k$ fixed points of
  $\pi_1$ and $n-m+y+k$ fixed points of $\pi_2$. 
We can complete prescribing the action of $\pi_1$ and $\pi_2$ in
$(m-k)!^2$ ways. 

The same total weight may also be found by fixing the pair
$(\pi_1,\pi_2)$ first, and summing over all {\em allowable pairs}
of colorings $(\gamma_1,\gamma_2)$. We call the pair
$(\gamma_1,\gamma_2)$ allowable, if the quadruplet
$(\pi_1,\pi_2,\gamma_1,\gamma_2)$ satisfies the conditions listed
above. Suppose $j_0$ is a common fixed point of $\pi_1$ and $\pi_2$. 
Observe that the contribution of all allowable pairs $(\gamma_1,\gamma_2)$ 
satisfying $\gamma_1(j_0)=\gamma_2(j_0)=2$ cancels the contribution of all
allowable pairs satisfying $\gamma_1(j_0)=\gamma_2(j_0)=1$. Indeed, 
let $(\gamma_1,\gamma_2)$ an allowable pair satisfying
$\gamma_1(j_0)=\gamma_2(j_0)>0$. Then the pair $(\gamma_1',\gamma_2')$
defined by 
$$
\gamma_i'(j)=
\begin{cases}
\gamma_i(j) & \mbox{if $j\neq j_0$} \\
3-\gamma_i(j) & \mbox{if $j= j_0$} \\
\end{cases}
$$
is also allowable and also satisfies
$\gamma'_1(j_0)=\gamma_2'(j_0)>0$. (Here $j\in \{1,\ldots,n\}$ for
$\gamma_1$ and $\gamma_1'$ and $j\in \{1,\ldots,n+y\}$ for $\gamma_2$
and $\gamma_2'$.) The map $(\gamma_1,\gamma_2)\mapsto
(\gamma_1',\gamma_2')$ is an involution that matches canceling terms:
the only difference between the respective contribution is a factor of
$x$ or $-x$ associated to $j_0$. Therefore we may restrict our attention
to the total weight of quadruplets
$(\pi_1,\pi_2,\gamma_1,\gamma_2)$ satisfying the following
additional criteria:
\begin{itemize}
\item[(5)] no $j$ satisfies $\gamma_1(j)=\gamma_2(j)=2$, in particular, the
  colorings $\gamma_1$ and $\gamma_2$ map into the set $\{0,1\}$;
\item[(6)] no $j$ satisfies $\gamma_1(j)=\gamma_2(j)=1$, in other words, the
  sets $\{j\in \{1,\ldots,n\}\::\: \gamma_1(j)=1\}$ and $\{j\in
  \{1,\ldots,n+y\}\::\: \gamma_2(j)=1\}$ are disjoint.
\end{itemize}
When computing the total weight of such quadruplets, we may first select
$k$ as the number of elements $j\in \{1,\ldots,n\}$ satisfying
$\gamma_1(j)=1$ and select them in $\binom{n}{k}$ ways. By rule
(\ref{it:exc}) there must be $y+k$ elements $j\in\{1,\ldots,n+y\}$
satisfying $\gamma_2(j)=1$ and, by rule (6), this set is disjoint of the
previously selected $k$-element set. Thus there are $\binom{n+y-k}{y+k}$
ways to select them. So far we have selected $k$ fixed points of $\pi_1$
and $y+k$ fixed points of $\pi_2$. There are $(n-k)!^2$ ways to
complete prescribing $\pi_1$ and $\pi_2$, the weight of the
quadruplet is $x^k$. We obtain a total contribution of 
$\sum_{k=0}^n \binom{n}{k}\binom{n+y-k}{y+k}(n-k)!^2 x^k$ which is
exactly $n!$ times the right and side.
\end{proof}
\begin{lemma}
\label{lem:ssum}
The polynomials $S_n(x,y)$ satisfy 
$$
\sum_{m=0}^n \frac{(-x)^{n-m}}{(n-m)!} S_m(x,y)
=\frac{1}{n!}\sum_{k=0}^{\lfloor (n-1)/2 \rfloor }
\binom{n}{k}\binom{n+y-k}{n-2k-1}(n-k)!(n-k-1)! x^{k+1}
$$
for all $n\geq 0$.
\end{lemma}
\begin{proof}
A proof may be obtained by performing slight modifications to the proof
of Lemma~\ref{lem:rsum}, which we outline below. The left hand side equals
$$
\frac{x}{n!}\sum_{m=0}^n \binom{n}{n-m} (-x)^{n-m} 
\sum_{k=0}^m \binom{m}{k} \binom{m+y}{y+k+1} (m-k)!(m-k-1)! x^k,
$$
which may be considered as $x/n!$ times the total weight of quadruplets
$(\pi_1,\pi_2,\gamma_1,\gamma_2)$, where the only change to the
definition is that, instead of (4), now we require
\begin{itemize}
\item[(4')] the colorings $\gamma_1$ and $\gamma_2$ satisfy $|\{j\in
  \{1,\ldots,n+y\}\::\: \gamma_2(j)>0\}|=|\{j\in \{1,\ldots,n\}\::\:
  \gamma_1(j)>0\}|+y+1$.
\end{itemize}
Thus we will have to select $y+k+1$ elements $j$ (instead of $y+k$)  
satisfying $\pi_2(j)=1$, in $\binom{m+y}{y+k+1}$. In the last stage, we
will have selected $n-m+y+k+1$ fixed points of $\pi_2$, thus we will
only have $(m-k-1)!$ ways to complete the selection of $\pi_2$. 

The same involution as before shows that we may again restrict our
attention to those quadruplets which satisfy the additional conditions
(5) and (6). Let us set $k$ again as the number of elements $j\in
\{1,\ldots,n\}$ satisfying $\gamma_1(j)=1$. The only adjustment we need
to make to the reasoning is to observe that now we need to have $y+k+1$
elements satisfying $|\gamma_2(j)=1|$ which may be selected in
$\binom{n+y-k}{y+k+1}$ ways, instead of $\binom{n+y-k}{y+k}$
ways. Finally, we may complete the selection of $\pi_2$ in $(n-k-1)!$
ways, instead of $(n-k)!$. The resulting total weight is exactly $x/n!$
times the right hand side of our stated equality. 
\end{proof}
For the sake of use in Section~\ref{sec:conv} let us note that
$R_n(x,y)$ and $S_n(x,y)$ may also be written in the following shorter
form, using falling factorials:
\begin{equation}
\label{eq:rxys}
R_n(x,y)=\sum_{k=0}^n \frac{x^k (y+n)_{n-k}}{k!} 
\quad\mbox{and}\quad
S_n(x,y)=\sum_{k=0}^{n-1}\frac{x^{k+1}(y+n)_{n-k-1}}{k!} 
\end{equation}
A similar simplification yields that Lemma~\ref{lem:rsum} is equivalent to 
\begin{equation}
\label{eq:rsum}
\sum_{m=0}^n \frac{(-x)^{n-m}}{(n-m)!} R_m(x,y)
=\sum_{k=0}^{\lfloor n/2 \rfloor} \frac{(n-k)!}{k!}\binom{n+y-k}{n-2k} x^k
\quad\mbox{for all $n\geq 0$}, 
\end{equation}
and that Lemma~\ref{lem:ssum} has the compact form
\begin{equation}
\label{eq:ssum}
\sum_{m=0}^n \frac{(-x)^{n-m}}{(n-m)!} S_m(x,y)
=\sum_{k=0}^{\lfloor (n-1)/2 \rfloor}
\frac{(n-k-1)!}{k!}\binom{n+y-k}{n-2k-1}x^{k+1}
\quad\mbox{for all $n\geq 0$}.
\end{equation}

\section{Calculating the convergents}
\label{sec:conv}

In this section we calculate $(nd+r-1)$st convergent of
$\xi(\alpha,\beta_0,\beta_1,d,r)$ directly, from the Euler-Mindig
formulas (\ref{eq:EM}). We begin by observing that, for an arbitrary
continued fraction $[a_0,a_1,\ldots]$ and an arbitrary positive integer
$\alpha$,  the numerator $p_n$  in (\ref{eq:EM}) may be rewritten as   
$$
p_n=\sum_{S\subseteq_e \{0,\ldots,n\}} \prod_{i\in S}(a_i-\alpha+\alpha)
=
\sum_{S\subseteq_e \{0,\ldots,n\}} \sum_{T\subseteq S}
\prod_{i\in T}(a_i-\alpha) \alpha^{|S\setminus T|}.
$$ 
Changing the order of summation gives
\begin{equation}
\label{eq:ashift}
p_n=\sum_{T\subseteq \{0,\ldots,n\}} \prod_{i\in T}
(a_i-\alpha)\sum_{T\subseteq S\subseteq_e \{0,\ldots,n\}}\alpha^{|S|}.
\end{equation}
For $\xi(\alpha,\beta_0,\beta_1,d,r)$ we have $a_i=\alpha$ unless $i$ is
congruent to $r$ modulo $d$. Thus, to compute $p_{nd+r-1}$ using
(\ref{eq:ashift}), we only need to sum over subsets $T$ whose elements
are all congruent to $r$ modulo $d$. This observation yields the
following recurrence:
\begin{equation}
\label{eq:prec}
p_{nd+r-1}=F_{nd+r+1}(\alpha) + 
\sum_{k=0}^{n-1} p_{kd+r-1}\cdot (\beta_0+\beta_1\cdot k-\alpha)\cdot 
F_{(n-k)d}(\alpha).  
\end{equation}
By Eq.\ (\ref{eq:Fq}), the term $F_{nd+r+1}(\alpha)$ above is the
contribution of $T=\emptyset$, whereas the term $p_{kd+r-1}\cdot
(\beta_0+\beta_1\cdot k-\alpha)\cdot F_{(n-k)d+1}(\alpha)$ is the total
contribution of all sets $T$ whose largest element is $kd+r$. 
Substituting $n=0$ into (\ref{eq:prec}) yields the initial condition
$$
p_{r-1}=F_{r+1}(\alpha)
$$
which is obviously true for $r>0$, and it is also valid when $r=0$ after
setting $p_{-1}=1$, as usual. Consider the formal Laurent series 
$$
y(t):=\sum_{n=0}^{\infty} p_{nd+r-1}\cdot t^{\beta_0+\beta_1\cdot
  n-\alpha}\in {\mathbb Q}((t)). 
$$
For $y(t)$, Eq.\ (\ref{eq:prec}) yields
\begin{equation}
\label{eq:yt1}
y(t):=\sum_{n=0}^{\infty} F_{nd+r+1}(\alpha) t^{\beta_1 n}\cdot
t^{\beta_0-\alpha}+t\cdot \sum_{n=1}^{\infty} F_{nd}(\alpha)
t^{\beta_1 n} \cdot y'(t),  
\end{equation}
where $y'(t)$ is the formal derivative of $y(t)$ with respect to $t$.
To write (\ref{eq:yt1}) in a more explicit form, observe that, by
Eq.\ (\ref{eq:Fnqc}), we have 
$$
\sum_{n=0}^{\infty} F_{nd+r+1}(\alpha) t^n
=\sum_{n=0}^{\infty}
 \frac{\rho_1^{nd+r+1}-\rho_2^{nd+r+1}}{\sqrt{\alpha^2+4}} t^n   
=\frac{1}{\sqrt{\alpha^2+4}} 
  \left(\frac{\rho_1^{r+1}}{1-\rho_1^d t}
    -\frac{\rho_2^{r+1}}{1-\rho_2^d t}\right). 
$$ 
Using the fact that $\rho_1\rho_2=-1$, the above equation may be
rewritten as 
\begin{eqnarray*}
\sum_{n=0}^{\infty} F_{nd+r+1}(\alpha) t^n
&=&\frac{1}{\sqrt{\alpha^2+4}} 
  \frac{\rho_1^{r+1}-\rho_2^{r+1} -
    (\rho_1^{r+1}\rho_2^d-\rho_2^{r+1}\rho_1^d)t}{1 -
    (\rho_1^d+\rho_2^d)t+(-1)^d t^2}\\ 
&=&
\frac{1}{\sqrt{\alpha^2+4}} 
  \frac{\rho_1^{r+1}-\rho_2^{r+1}+(-1)^{r+1}
    (\rho_1^{d-r-1}-\rho_2^{d-r-1})t}{1 -
    (\rho_1^d+\rho_2^d)t+(-1)^d t^2}.\\ 
\end{eqnarray*} 
By Eqs.\ (\ref{eq:Fq}) and (\ref{eq:Lq}), we obtain
\begin{equation}
\label{eq:Fndr}
\sum_{n=0}^{\infty} F_{nd+r+1}(\alpha) t^n=
\frac{F_{r+1}(\alpha)+(-1)^{r+1}F_{d-r-1}(\alpha)t}
{1 - L_d(\alpha)t+(-1)^d t^2}. 
\end{equation}
Note that substituting $r=d-1$ in Eq.\ (\ref{eq:Fndr}) yields
$$
\sum_{n=1}^{\infty} F_{nd}(\alpha) t^{n}
= t\cdot \sum_{n=0}^{\infty} F_{nd+d}(\alpha) t^{n}
=\frac{F_{d}(\alpha)\cdot t}
{1 - L_d(\alpha) t+(-1)^d t^2},
$$
since $F_{0}(\alpha)=0$. Using these last two
equations, we may rewrite (\ref{eq:yt1}) as 
$$
y(t)=\frac{F_{r+1}(\alpha)+(-1)^{r+1}F_{d-r-1}(\alpha)t^{\beta_1}}
{1 - L_d(\alpha) t^{\beta_1}+(-1)^d t^{2\beta_1}}
\cdot
t^{\beta_0-\alpha}+ 
\frac{F_{d}(\alpha)\cdot t^{\beta_1+1}}
{1 - L_d(\alpha) t^{\beta_1}+(-1)^d t^{2\beta_1}} \cdot y'(t).
$$
Rearranging to express the derivative of $y(t)$ yields
\begin{equation}
\label{eq:ydiff}
\begin{array}{rcl}
y'(t)& = & \di \frac{y(t)\cdot (t^{-\beta_1-1} -
  L_d(\alpha) t^{-1}+(-1)^d 
t^{\beta_1-1})}{F_{d}(\alpha)}\\
&& \di  -\frac{(F_{r+1}(\alpha)+(-1)^{r+1}F_{d-r-1}(\alpha)t^{\beta_1})\cdot
t^{\beta_0-\alpha}}{F_{d}(\alpha)\cdot t^{\beta_1+1}}. 
\end{array}
\end{equation}
Inspired by the way we solve ordinary differential equations in
analysis, we will {\em guess} the solution of (\ref{eq:ydiff}) by
``solving'' first the corresponding homogeneous equation and then
replace the arbitrary constant by a formal Laurent series. The next few
lines will not make sense, they just indicate how one may come up
with a good guess for $y(t)$. A reader who does not like ``obscure
reasoning,'' should skip ahead to (\ref{eq:yz}) and accept that there is
no ``rational explanation'' as to why introducing the formal Laurent
series $z(t)$ is a good idea.

The homogeneous equation 
$$
y_H'(t)= \frac{y_H(t)\cdot (t^{-\beta_1-1} -
  L_d(\alpha) t^{-1}+(-1)^d 
t^{\beta_1-1})}{F_{d}(\alpha)}
$$
``may be rewritten as'' 
$$
\frac{d}{dt}\ln(y_H(t))=\frac{t^{-\beta_1-1} -
  L_d(\alpha) t^{-1}+(-1)^d 
t^{\beta_1-1}}{F_{d}(\alpha)},
$$
``yielding'' 
$$
y_H(t)=C\cdot t^{-L_d(\alpha)/F_{d}(\alpha)}
\exp\left(\frac{-t^{-\beta_1} + (-1)^d t^{\beta_1}}{\beta_1\cdot
  F_{d}(\alpha)}\right). 
$$
where $C$ is an arbitrary constant. This ``solution'' to the ``homogeneous
equation'' suggests looking for a solution to Eq.\ (\ref{eq:ydiff}) of the form
$$
y(t)=z(t)\cdot t^{-L_d(\alpha)/F_{d}(\alpha)}
\exp\left(\frac{-t^{-\beta_1} + (-1)^d t^{\beta_1}}{\beta_1\cdot
  F_{d}(\alpha)}\right).
$$
Equivalently, we would want to set 
$$
z(t):=y(t)\cdot t^{L_d(\alpha)/F_{d}(\alpha)}
\exp\left(\frac{t^{-\beta_1}+(-1)^{d-1} t^{\beta_1}}{\beta_1\cdot
  F_{d}(\alpha)}\right). 
$$
Alas, the resulting formal expression would contain arbitrary large
positive and as well as arbitrary small negative powers of $t$, it does
not resemble a Laurent series at all. Hoping that a slight change would
not upset our calculations irreparably, we {\em define} $z(t)$ by
setting 
\begin{equation}
\label{eq:yz}
z(t):=y(t)\cdot t^{L_d(\alpha)/F_{d}(\alpha)}
\exp\left(\frac{(-1)^{d-1} t^{\beta_1}}{\beta_1\cdot
  F_{d}(\alpha)}\right). 
\end{equation}
Note that $z(t)$ is a formal Laurent series in the variable 
$t^{1/F_{d}(\alpha)}$, an infinite formal sum of the form 
$$
z(t)=\sum_{n=0}^{\infty} s_n t^{\beta_1\cdot n+\beta_0-\alpha
  +L_d(\alpha)/F_{d}(\alpha)}\in {\mathbb Q}((t^{1/F_{d}(\alpha)})).
$$
Taking the derivative on both sides of (\ref{eq:yz}) yields
\begin{eqnarray*}
z'(t)&=&y'(t)\cdot t^{\frac{L_d(\alpha)}{F_{d}(\alpha)}}
\exp\left(\frac{(-1)^{d-1} t^{\beta_1}}{\beta_1\cdot
  F_{d}(\alpha)}\right)\\
&& + y(t) \cdot t^{\frac{L_d(\alpha)}{F_{d}(\alpha)}}
\exp\left(\frac{(-1)^{d-1} t^{\beta_1}}{\beta_1\cdot
  F_{d}(\alpha)}\right)\cdot
\left(\frac{L_d(\alpha) t^{-1}-(-1)^{d}
  t^{\beta_1-1}}{F_{d}(\alpha)} \right)\\ 
&=& t^{\frac{L_d(\alpha)}{F_{d}(\alpha)}}
\exp\left(\frac{(-1)^{d-1} t^{\beta_1}}{\beta_1\cdot
  F_{d}(\alpha)}\right)\cdot \left(y'(t)+y(t)\cdot
\frac{L_d(\alpha) t^{-1}-(-1)^{d}
  t^{\beta_1-1}}{F_{d}(\alpha)}\right).
\end{eqnarray*}
(Note that $L_d(\alpha)/F_{d}(\alpha)$ is always positive.)
After substituting the value of $y'(t)$ from (\ref{eq:ydiff}) and
simplifying, we obtain
$$
z'(t)=t^{\frac{L_d(\alpha)}{F_{d}(\alpha)}}
\exp\left(\frac{(-1)^{d-1} t^{\beta_1}}{\beta_1\cdot
  F_{d}(\alpha)}\right)\cdot 
\frac{y(t)-(F_{r+1}(\alpha)+(-1)^{r+1} F_{d-r-1}(\alpha)t^{\beta_1})t^{\beta_0-\alpha}}{F_{d}(\alpha)t^{\beta_1+1}}.   
$$
By (\ref{eq:yz}), the last equation is equivalent to 
\begin{eqnarray*}
z'(t)&=&\frac{z(t)}{F_{d}(\alpha)t^{\beta_1+1}}\\
&&- t^{\frac{L_d(\alpha)}{F_{d}(\alpha)}}
\exp\left(\frac{(-1)^{d-1} t^{\beta_1}}{\beta_1\cdot
  F_{d}(\alpha)}\right)\cdot 
\frac{(F_{r+1}(\alpha)+(-1)^{r+1}F_{d-r-1}(\alpha)t^{\beta_1})t^{\beta_0-\alpha}}{F_{d}(\alpha)t^{\beta_1+1}}.\\  
\end{eqnarray*}
Comparing the coefficients of $t^{\beta_1\cdot n+\beta_0-\alpha
  +L_d(\alpha)/F_{d}(\alpha)-1}$  on both sides yields
\begin{eqnarray*}
s_n\cdot\left(\beta_1\cdot n+\beta_0-\alpha +\frac{L_d(\alpha)}{F_{d}(\alpha)}\right)
&=&
\frac{s_{n+1}}{F_{d}(\alpha)}
-
\frac{F_{r+1}(\alpha)}{F_{d}(\alpha)}\cdot\frac{\left(\frac{(-1)^{d-1}
    }{\beta_1\cdot F_{d}(\alpha)}\right)^{n+1}}{(n+1)!}\\
&&+
\frac{(-1)^{r}F_{d-r-1}(\alpha)}{F_{d}(\alpha)}\cdot\frac{\left(\frac{(-1)^{d-1} 
    }{\beta_1\cdot F_{d}(\alpha)}\right)^{n}}{n!}. 
\end{eqnarray*}
Note that the factor $(\beta_1\cdot n+\beta_0-\alpha
+L_d(\alpha)/F_{d}(\alpha))$ in the last equation equals
$\beta_1\cdot(n+\sigma)$, where $\sigma$ is the magic sum, and that the
magic quotient $\rho$ appears twice on the right hand side.
Since, by Corollary~\ref{cor:spos},
$(\sigma+n)_{n+1}$ is not zero, we may divide both sides by 
$F_{d}(\alpha)^n\beta_1^{n+1}(\sigma+n)_{n+1}$, and obtain the 
following recurrence for $\widetilde{s}_n:=s_n/(F_{d}(\alpha)^{n}\beta_1^{n}(\sigma+n-1)_{n})$:
$$
\widetilde{s}_{n+1}
=
\widetilde{s}_{n}
+
F_{r+1}(\alpha)\cdot\frac{\rho^{n+1}}{(n+1)!(\sigma+n)_{n+1}}
-
\frac{(-1)^{r}F_{d-r-1}(\alpha)}{F_{d}(\alpha)\beta_1}\cdot\frac{\rho^{n}}{n!(\sigma+n)_{n+1}}. 
$$
Considering the fact that $\widetilde{s}_0=s_0=F_{r+1}(\alpha)$
and that
$(-1)^{r}F_{d-r-1}(\alpha)/(F_{d}(\alpha)\beta_1)
=
(-1)^{d-1-r}F_{d-r-1}(\alpha)F_{d}(\alpha)\beta_1\cdot\rho$,
the last recurrence implies
$$
\widetilde{s}_n
=\di
F_{r+1}(\alpha)\cdot
\sum_{k=0}^{n}
\frac{\rho^{k}}{k!(\sigma+k-1)_{k}}
+
(-1)^{d-r}F_{d-r-1}(\alpha)F_{d}(\alpha)\beta_1\cdot
\sum_{k=0}^{n-1}\frac{\rho^{k+1}}{k!(\sigma+k)_{k+1}}.
$$
Multiplying both sides by $(\sigma+n)_{n+1}$ yields
$$
\begin{array}{rcl}
\di \frac{s_{n}}{F_{d}(\alpha)^{n}\beta_1^{n}}
&=&\di
F_{r+1}(\alpha)\cdot
\sum_{k=0}^{n}
\frac{\rho^{k}(\sigma+n)_{n-k}}{k!}\\
&&\di +
(-1)^{d-r}F_{d-r-1}(\alpha)F_{d}(\alpha)\beta_1\cdot
\sum_{k=0}^{n-1}\frac{\rho^{k+1}(\sigma+n)_{n-k-1}}{k!}.\\
\end{array}
$$
By the formulas given in Eq.\ (\ref{eq:rxys}), we may rewrite the previous
equation in terms of the polynomials $R_n(x,y)$ and $S_n(x,y)$ as follows:
\begin{equation}
\label{eq:sRS}
\di \frac{s_n}{F_{d}(\alpha)^n\beta_1^n}
=
\di F_{r+1}(\alpha)
R_n\left(\rho,\sigma-1\right)
+(-1)^{d-r}F_{d-r-1}(\alpha)F_{d}(\alpha)\beta_1\cdot
S_n\left(\rho,\sigma-1\right).
\end{equation}
By (\ref{eq:yz}) we have 
$$
y(t)= z(t)\cdot t^{-L_d(\alpha)/F_{d}(\alpha)}
\exp\left(-\frac{(-1)^{d-1} t^{\beta_1}}{\beta_1\cdot
  F_{d}(\alpha)}\right). 
$$
Comparing the coefficients of $t^{\beta_0+\beta_1\cdot n-\alpha}$ yields
$$
p_{nd+r-1}
=\sum_{m=0}^n s_m\cdot 
    \frac{\left(\frac{(-1)^{d-1}}{F_{d}(\alpha)\beta_1}\right)^{n-m}}{(n-m)!},
$$
hence we have 
$$
\frac{p_{nd+r-1}}{F_{d}(\alpha)^n\beta_1^n}
=
 \sum_{m=0}^n
   \frac{s_m}{F_{d}(\alpha)^m\beta_1^m}\cdot 
   \frac{\left(\frac{(-1)^{d-1}}{F_{d}(\alpha)^2\beta_1^2}\right)^{n-m}}
   {(n-m)!}
=
 \sum_{m=0}^n
   \frac{s_m}{F_{d}(\alpha)^m\beta_1^m}\cdot 
   \frac{\rho^{n-m}}{(n-m)!}.
$$
Substituting Eq.\ (\ref{eq:sRS}) into this last equation and using
Eqs.\ (\ref{eq:rsum}) and (\ref{eq:ssum}) yields the formula for $p_{nd+r-1}$
in Theorem~\ref{thm:hurwitzc}.

%%%%%%%%%%%%%%%
For positive $r$, the formula for $q_{nd+r-1}$ stated
in Theorem~\ref{thm:hurwitzc} is an easy consequence of the formula for
$p_{nd+r-1}$. Indeed, by Lemma~\ref{lem:shift}, the denominator
$q_{nd+r-1}$ is the same as the numerator $p_{nd+r-2}$ 
associated to $\xi(\alpha,\beta_0,\beta_1,d,r-1)$, thus we only need to
replace $r$ by $r-1$ in the formula stated for $p_{nd+r-2}$.
It only remains to show the following lemma.
\begin{lemma}
\label{lem:r=0}
The equation stated for $q_{nd+r-1}$ in Theorem~\ref{thm:hurwitzc}
remains valid when we substitute $r=0$.
\end{lemma}
\begin{proof}
Since $F_0(\alpha)=0$, substituting $r=0$ in Theorem~\ref{thm:hurwitzc} gives 
$$
\frac{q_{nd-1}}{F_{d}(\alpha)^n\beta_1^n}
=
(-1)^{d+1}F_{d}(\alpha)^2\cdot \beta_1 
\sum_{k=0}^{\lfloor (n-1)/2 \rfloor}
\frac{(n-k-1)!}{k!}\binom{n+\sigma-1-k}{n-2k-1}\rho^{k+1}.
$$
Here we may replace $(-1)^{d+1}F_{d}(\alpha)^2\cdot \beta_1$ by
$\rho^{-1}\beta_1^{-1}$. After rearranging, we obtain 
\begin{equation}
\label{eq:r=0}
q_{nd-1}
=
F_{d}(\alpha)^n\beta_1^{n-1}
\sum_{k=0}^{\lfloor (n-1)/2 \rfloor}
\frac{(n-k-1)!}{k!}\binom{n+\sigma-1-k}{n-2k-1}\rho^{k}.
\end{equation}
We need to show the validity of this equation. Observe that 
\begin{eqnarray*}
\xi(\alpha,\beta_0,\beta_1,d,0)
&=&
[\beta_0,\underbrace{\alpha,\ldots,\alpha}_{d-1},
  \beta_0+\beta_1,\ldots]
=\beta_0-\alpha+[\underbrace{\alpha,\ldots,\alpha}_{d},
  \beta_0+\beta_1,\ldots]\\
&=&\beta_0-\alpha+\xi(\alpha,\beta_0+\beta_1,\beta_1,d,d)
\end{eqnarray*}
Thus the denominator $q_{nd-1}$ associated to
$\xi(\alpha,\beta_0,\beta_1,d,0)$ is the same 
as the same as the denominator $q{(n-1)d-1}$ associated to 
$\xi(\alpha,\beta_0+\beta_1,\beta_1,d,d)$. We may apply the already shown part
of Theorem~\ref{thm:hurwitzc}. Since the magic quotient $\rho$ depends
only on $\alpha$, $\beta_1$ and $d$, it is the same for
$\xi(\alpha,\beta_0,\beta_1,d,0)$ and for 
$\xi(\alpha,\beta_0+\beta_1,\beta_1,d,d)$. For the magic sums we get
$$
\sigma(\alpha,\beta_0+\beta_1,\beta_1,d,d)
=
\frac{\beta_0+\beta_1-\alpha}{\beta_1}+\frac{L_d(\alpha)}{\beta_1
  F_{d}(\alpha)}
=\sigma(\alpha,\beta_0,\beta_1,d,0)+1.  
$$ 
Therefore we may obtain an equation for the $q_{nd-1}$ associated to
$\xi(\alpha,\beta_0,\beta_1,d,0)$ by replacing  $n$ with $n-1$, $r$ with
$d$ and $\sigma$ with $\sigma+1$ in the formula for $q_{nd-1}$ in
Theorem~\ref{thm:hurwitzc}. Since $F_0(\alpha)=0$, the second sum
vanishes and we get 
$$
\frac{q_{nd+r-1}(\alpha,\beta_0,\beta_1,d,0)}{F_{d}(\alpha)^{n-1}\beta_1^{n-1}}
=
\di F_{d}(\alpha)
\sum_{k=0}^{\lfloor (n-1)/2\rfloor}
\frac{(n-1-k)!}{k!}\binom{n+\sigma-1-k}{n-1-2k} \rho^k,
$$
which is obviously equivalent to (\ref{eq:r=0}).
\end{proof}

\section{Special cases leading to elementary expressions}
\label{sec:examples}

In this section we describe all instances of Theorem~\ref{thm:hurwitzn}
for which the magic sum $\sigma$ is the half of an odd integer,
forcing all (modified) Bessel functions in the statement to be known
elementary functions. There is not much to say about the case when the
quasi-period $d$ is given by $d=1$:
as seen in Remark~\ref{rem:Lehmersway}, we have
$\sigma=\beta_0/\beta_1$ thus $\sigma$ depends only on $\beta_0$ and
$\beta_1$ in a very simple fashion. 
In the case when $d\geq 2$ we give a similarly simple description.
\begin{theorem}
\label{thm:elementary}
If $d\geq 2$ then $\sigma$ is the half of an odd integer if and only
if one of the following conditions holds:
\begin{enumerate}
\item $d=3$, $\alpha=1$ and $(\beta_0+1)/\beta_1$ is the half of an odd
  integer;
\item $d=2$, $\alpha=1$ and $(\beta_0+2)/\beta_1$ is the half of an odd
  integer;
\item $d=2$, $\alpha=2$ and $(\beta_0+1)/\beta_1$ is the half of an odd
  integer;
\item $d=2$, $\alpha=4$ and $(2\beta_0+1)/\beta_1$ is an integer.
\end{enumerate}  
\end{theorem}
\begin{proof}
First we show that even assuming that $\sigma$ is the half of an (even or odd)
integer implies that $d$ is at most $3$. The fact that $\sigma$ belongs
to $(1/2)\cdot {\mathbb Z}$ implies the same for $\beta_1\sigma
+\alpha-\beta_0=L_d(\alpha)/F_d(\alpha)$. We may subtract any
integer multiple of $F_d(\alpha)$ from $L_d(\alpha)$ and still have element of 
$(1/2)\cdot {\mathbb Z}$. Let us select $r_d(\alpha):=L_d(\alpha)-\alpha
F_d(\alpha)$ and consider the fraction $r_d(\alpha)/F_d(\alpha)\in
(1/2)\cdot {\mathbb Z}$. We claim that 
\begin{equation}
\label{eq:r<1}
0<\frac{r_d(\alpha)}{F_d(\alpha)}<1
\quad
\mbox{holds for $\alpha\geq 2$ and $d\geq 3$.}
\end{equation}
Indeed, for $d=3$, we have 
$r_3(\alpha)
=L_3(\alpha)-\alpha F_3(\alpha)=(\alpha^3+3\alpha)-\alpha(\alpha^2+1)
=2\alpha$
and $F_3(\alpha)=\alpha^2+1$. Both $r_3(\alpha)$ and $F_3(\alpha)$ are
positive and $r_3(\alpha)<F_3(\alpha)$ follows from 
$$
F_3(\alpha)-r_3(\alpha)=(\alpha-1)^2>0 \quad \mbox{for $\alpha\geq 2$.}
$$
For $d=4$ we have 
$
r_4(\alpha)=L_4(\alpha)-\alpha F_4(\alpha)
=\alpha^4+4\alpha^2+2-\alpha(\alpha^3+2\alpha)
=2\alpha^2+2
$ and $F_4(\alpha)=\alpha^3+2\alpha$. Both $r_4(\alpha)$ and $F_4(\alpha)$ are
positive and $r_4(\alpha)<F_4(\alpha)$ follows from 
$$
F_4(\alpha)-r_4(\alpha)
=\alpha^3-2\alpha^2+2\alpha-2
=\alpha^2(\alpha^1-1)+2(\alpha-1)
>0 \quad \mbox{for $\alpha\geq 2$.}
$$
For larger values of $d$, we may show that 
$$
0<r_d(\alpha)<F_d(\alpha) \quad\mbox{holds when $d\geq 3$ and $\alpha\geq 2$.}
$$
by induction on $d$, using the fact that the statement is valid for
$d\in \{3,4\}$ and that $r_d(\alpha)$ satisfies the same recurrence as
$F_d(\alpha)$, allowing to express the inequality for the next value of
$d$ as a positive combination of the inequalities for the current and
the previous values of $d$. This concludes the proof of (\ref{eq:r<1}).

As a consequence of the inequality (\ref{eq:r<1}), whenever $d\geq 3$ and
$\alpha\geq 2$, the only way for $r_d(\alpha)/F_d(\alpha)$ to be an
integer is to have $r_d(\alpha)/F_d(\alpha)=1/2$. In other words, in
this case, we must have $2r_d(\alpha)=F_d(\alpha)$ which is equivalent
to $2L_d(\alpha)=(2\alpha+1)F_d(\alpha)$. To show that $\alpha\geq 2$
and $d\geq 3$ can not hold simultaneously, it suffices to show that
$2L_d(\alpha)$ can never be equal to $(2\alpha+1)F_d(\alpha)$. To do so,
first we observe that 
$$
2L_d(\alpha) > (2\alpha+1)F_d(\alpha) \quad\mbox{holds for
  $\alpha\in\{2,3\}$ and $d\geq 3$.} 
$$
Indeed, for $\alpha=2$ we have $28=2L_3(2)>(2\cdot 2+1)F_3(2)=25$
and $68=2L_4(2)>(2\cdot 2+1)F_4(2)=60$ and we may prove the same
inequality for higher values of $d$ by induction. Similarly, for
$\alpha=3$ we have $72=2L_3(3)>(2\cdot 3+1)F_3(3)=70$
and $238=2L_4(3)>(2\cdot 3+1)F_4(3)=231$ and we may proceed again by
induction on $d$. Finally, to exclude $\alpha\geq 4$, we will show
$$
2L_d(\alpha) < (2\alpha+1)F_d(\alpha) \quad\mbox{holds for
  $\alpha\geq 4$ and $d\geq 3$.} 
$$
For $d=3$ we have 
$$
(2\alpha+1)F_3(\alpha)-2L_3(\alpha)
=
%%(2\alpha+1)(\alpha^2+1)-2(\alpha^3+3\alpha)
%%=
\alpha^2-4\alpha+1
=\alpha(\alpha-4)+1>0,
$$
and for $d=4$ we have 
$$
(2\alpha+1)F_4(\alpha)-2L_4(\alpha)
=
%%(2\alpha+1)(\alpha^3+2\alpha)-2(\alpha^4+4\alpha^2+2)
%%=
\alpha^3-4\alpha^2+2\alpha-4
=\alpha^2(\alpha-4)+2\alpha(\alpha-2)>0,
$$
and again we may proceed by induction on $d$.

We obtained that, for $d\geq 3$,  $\sigma\in (1/2)\cdot {\mathbb Z}$ is
only possible if $\alpha=1$. In that case, for $d\geq 4$ we have
$L_d(1)-2F_d(1)=F_{d-3}(1)$ (this may be shown by induction). Clearly
$L_d(1)/F_d(1)$ is the half of an integer, if and only if the same holds
for $F_{d-3}(1)/F_d(1)$. Now we may exclude the case $d\geq 4$
completely, after observing that
$$
0<\frac{F_{d-3}(1)}{F_d(1)}<\frac{1}{2}\quad\mbox{holds for $d\geq 4$}.
$$
Indeed, the above inequality is equivalent to $0<2F_{d-3}(1)<F_d(1)$
which may be easily shown by induction.

We have shown that $\sigma$ can only be the half of an integer if $d=2$
or $d=3$. In the case, when $d=3$, we have also shown that only
$\alpha=1$ is possible, and we get 
$$
\sigma
=\frac{\beta_0-1}{\beta_1}+\frac{L_3(1)}{\beta_1 F_{3}(1)} 
=\frac{\beta_0-1}{\beta_1}+\frac{4}{2\beta_1}
=\frac{\beta_0+1}{\beta_1}. 
$$ 
Consider finally the case when $d=2$. As before, $\sigma\in (1/2)
{\mathbb Z}$ implies that $L_d(2)/F_d(2)=(\alpha^2+2)/\alpha \in (1/2)
{\mathbb Z}$. This implies that $\alpha$ must be a divisor of $4$, that
is, an element of $\{1,2,4\}$. We have
$$
\sigma
=\frac{\beta_0-\alpha}{\beta_1}+\frac{L_2(\alpha)}{\beta_1 F_{2}(\alpha)} 
=\frac{\alpha\beta_0-\alpha^2}{\beta_1 \alpha}+\frac{\alpha^2+2}{\beta_1
  \alpha}  
=\frac{\alpha\beta_0+2}{\beta_1\alpha}.
$$
Therefore, for $\alpha=1$ we get
$\sigma=(\beta_0+2)/\beta_1$, for $\alpha=2$ we get
$\sigma=(\beta_0+1)/\beta_1$ and for $\alpha=4$ we get
$\sigma=(2\beta_0+1)/(2\beta_1)$.  Note that, in the case when
$\alpha=4$, $2\sigma=(2\beta_0+1)/\beta_1$ is necessarily odd, if it is
an integer.  
\end{proof}

A nice example of the case when $d=3$ and $\alpha=1$ in
Theorem~\ref{thm:elementary} above is the case when 
$d=3$, $\alpha=1$, $r=2$, $\beta_0=3m-1$, $\beta_1=2m$, 
for some $m>0$. In this case we get $\sigma=3/2$
and $\rho=1/(16m^2)$.
Theorem~\ref{thm:hurwitzn} gives 
$$
\xi(1,3m-1,2m,3,2)=
\frac{\di 2
I_{1/2}(1/(2m))}
{\di I_{1/2}(1/(2m))
+ I_{3/2}(1/(2m))}
$$
Using (\ref{eq:I1/2}) and  (\ref{eq:IJ3/2}) we may rewrite  the
preceding equation as 
$$
[\overline{1,1,3m-1+2mn}]_{n=0}^{\infty}=
\frac{2 \sinh (1/(2m))}{\cosh(1/(2m))-(2m-1)\sinh(1/(2m))}.
$$
For $m=1$, we obtain
$$[1,1,2,1,1,4,\ldots]=
\frac{e^{1/2}-e^{-1/2}}{e^{-1/2}}=e-1.
$$
A similarly nice example for the case when $d=2$ and $\alpha=1$ in
Theorem~\ref{thm:elementary} above is  
$$
\xi(1,3m-2,2m,2,1)
=
[\overline{1,3m-2+2mn}]_{n=0}^{\infty}
$$
for some $m>0$. In this example $\sigma=3/2$ and 
$\rho=(-1)/(4m^2)$ hold. Theorem~\ref{thm:hurwitzn} gives 
$$
\xi
=
\frac{J_{1/2}(2\sqrt{-\rho})}
{\di J_{1/2}(2\sqrt{-\rho})
- J_{3/2}(2\sqrt{-\rho})}
=
\frac{J_{1/2}(1/m)}
{\di J_{1/2}(1/m)
- J_{3/2}(1/m)}.
$$
Using (\ref{eq:J1/2}) and (\ref{eq:IJ3/2}) we may rewrite $\xi$ above as 
$$
\xi=\frac{\sin(1/m)}
{\di \cos(1/m) - (m-1)\sin(1/m)}.
$$ 
Substituting $m=1$ yields $\tan(1)=[1,1,1,3,1,5,1,7,\ldots]$.

We conclude this section with an example which is not likely to be
found in the literature, due to its ``sheer ugliness''. Let us set $\alpha=4$,
$\beta_0=7m+3$, $\beta_1=2m+1$, $d=2$ and $r=1$, where $m$ is any
nonnegative integer. For this example we
have $\sigma=7/2$ and $\rho=(-1)/16(2m+1)^2$, yielding
$2\sqrt{-\rho}=1/(4m+2)$. Using the fact that  
$$
J_{5/2}(z)=\sqrt{\frac{2}{\pi z}}
\left(\left(\frac{3}{z^2}-1\right)\sin(z)
-\frac{3}{z} \cos (z)\right)
$$
and
$$
J_{7/2}(z)=\sqrt{\frac{2}{\pi z}}
\left(\left(\frac{15}{z^3}-\frac{6}{z}\right)\sin(z)
-\left(\frac{15}{z^2}-1\right) \cos (z)\right)
$$
(see \cite[List of formul\ae: 46, 47]{McLachlan}),
Theorem~\ref{thm:hurwitzn} gives
$$
\xi=\frac{4 \left((12(2m+1)^2-1) \sin\left(\frac{1}{4m+2}\right)
-6(2m+1) \cos
\left(\frac{1}{4m+2}\right)\right)}
{(240m^2+228m+53) \cos
  \left(\frac{1}{4m+2}\right)-(960m^3+1392m^2+648m+97)
  \sin\left(\frac{1}{4m+2}\right)}. 
$$
The above formula was calculated from Theorem~\ref{thm:hurwitzn} with
the help of Maple. The same program was used to double-check its correctness
for selected values of $m$. For example, for $m=0$, we obtain
$$
\frac{4 \left(11 \sin\left(\frac{1}{2}\right)
-6\cos
\left(\frac{1}{2}\right)\right)}
{53 \cos
  \left(\frac{1}{2}\right)-97
  \sin\left(\frac{1}{2}\right)}=[4, 3, 4, 4, 4, 5, 4, 6, 4, 7, 4,\ldots ].
$$

\section{Special cases involving an integer magic sum}
\label{sec:examples2}

Another interesting special instance of our main result is the case
when the magic sum $\sigma$ is an integer. Specializing 
Theorem~\ref{thm:hurwitzn} to this case seems to yield less exciting
formulas, as (modified) Bessel functions of integer order are not known to
be elementary, even though there is a substantial literature on how to
compute them. On the other hand, the binomial coefficients appearing in
Theorem~\ref{thm:hurwitzc} are all ordinary (not generalized) binomial
coefficients. Motivated by this observation, we state the following
variant of Theorem~\ref{thm:elementary}. 

\begin{theorem}
\label{thm:integersig}
If $d\geq 2$ then $\sigma$ is an integer if and only
if one of the following conditions holds:
\begin{enumerate}
\item $d=3$, $\alpha=1$ and $(\beta_0+1)/\beta_1$ is an integer;
\item $d=2$, $\alpha=1$ and $(\beta_0+2)/\beta_1$ is an integer;
\item $d=2$, $\alpha=2$ and $(\beta_0+1)/\beta_1$ is an integer.
\end{enumerate}  
\end{theorem}
\begin{proof}
If $\sigma$ is an integer than it is also half of an integer, and 
we have shown in the proof of Theorem~\ref{thm:elementary} that this
implies $d=2$ or $d=3$. Furthermore, in the case when $d=3$, only
$\alpha=1$ is possible. We omit the completely analogous analysis of the
resulting finitely many cases. We only highlight the
fact that, as noted at the end of the proof of
Theorem~\ref{thm:elementary}, in the case when $d=2$ and $\alpha=4$, the
number $2\sigma$ can not be an even integer, thus $\sigma$ can not
be an integer.
\end{proof}

A nice example of the case when $d=3$ and $\alpha=1$ in
Theorem~\ref{thm:integersig} above is the case when 
$d=3$, $\alpha=1$,  $r=2$, $\beta_0=m-1$ and $\beta_1=m$ for some
$m>1$. In this case we get $\sigma=1$ and
$\rho=(2m)^{-2}$. 
Introducing
\begin{equation}
\label{eq:Pdef}
P_n(x)=\sum_{k=0}^{\lfloor (n-1)/2 \rfloor}
\frac{(n-k-1)!}{k!}\binom{n-k}{n-2k-1}x^{k+1},\quad\mbox{and}\quad
\end{equation}
\begin{equation}
\label{eq:Qdef}
Q_n(x)=\sum_{k=0}^{\lfloor n/2\rfloor}
\frac{(n-k)!}{k!}\binom{n-k}{n-2k} x^{k},
\end{equation}
Theorem~\ref{thm:hurwitzc} may be written in the form
\begin{equation}
\label{eq:pconv}
p_{3n+1}=2(2m)^n Q_n\left(\frac{1}{4m^2}\right) \quad\mbox{and}\quad
\end{equation} 
\begin{equation}
\label{eq:qconv}
q_{3n+1}=(2m)^n\left(2m P_n\left(\frac{1}{4m^2}\right)
+Q_n\left(\frac{1}{4m^2}\right)\right). 
\end{equation} 
The coefficients of the polynomials $P_n(x)$ and $Q_n(x)$, respectively, are
listed as sequences A221913 and A084950 in \cite{OEIS}. The quotient
$P_n(x)/Q_n(x)$ is the generalized continued continued fraction of the form
(\ref{eq:gcf}) where $a_i=i$ for all $0\leq i\leq n$ and $b_j=x$ for all
$1\leq j\leq n$. Entry A084950 states the formula 
\begin{equation}
\label{eq:wlang}
\lim_{n\rightarrow \infty} \frac{P_n(x)}{Q_n(x)}
=\sqrt{x}\cdot \frac{I_1(2\sqrt{x})}{I_0(2\sqrt{x})}
\end{equation}
with a proof outlined by Wolfdieter Lang. Eqs.\ (\ref{eq:pconv}) and
(\ref{eq:qconv}), combined with Theorem~\ref{thm:hurwitzn}, yield the
equation 
$$
\lim_{n\rightarrow \infty} \frac{2Q_n\left(\frac{1}{4m^2}\right)}{2m
  P_n\left(\frac{1}{4m^2}\right) 
+Q_n\left(\frac{1}{4m^2}\right)}=\frac{2I_0\left(\frac{1}{m}\right)}{I_0\left(\frac{1}{m}\right)+I_1\left(\frac{1}{m}\right)}.   
$$ 
After dividing the numerator and the denominator of the fraction on the
left hand side by $Q_n(x)$ we may easily deduce from the last equation 
that 
$$
\lim_{n\rightarrow \infty} \frac{P_n\left(\frac{1}{4m^2}\right)}{ 
Q_n\left(\frac{1}{4m^2}\right)}=\frac{1}{2m}\frac{I_1\left(\frac{1}{m}\right)}{I_0\left(\frac{1}{m}\right)},
\quad\mbox{holds for $m>1$.}   
$$ 
This is a special instance of (\ref{eq:wlang}) above for real numbers 
  $x$ of the form $x=1/(4m^2)$, where $m>1$ is an integer. It seems
possible that, with some effort, the definition of the continued fraction
$[1,1,m-1,1,1,2m-1,1,1,3m-1,\ldots]$ may be generalized to and shown
convergent for an arbitrary positive real $m$ and then, after extending
the validity of our main results to this setting, we may obtain another
(certainly more complicated) proof of (\ref{eq:wlang}). Departing on
this journey for the sake of this single formula is certainly not worth
the effort. On the other hand, it seems worth looking at in the future,
why each third convergent of the continued fraction
$[1,1,m-1,1,1,2m-1,1,1,3m-1,\ldots]$ can be matched up with a convergent
of the generalized continued fraction  $x/(1+x/(2+x/3+\cdots$ for
$x=1/(4m^2)$ and how far this ``coincidence'' could be generalized.

\section*{Acknowledgments}
This work was partially supported by a grant from the Simons Foundation
(\#245153 to G\'abor Hetyei). The author is indebted to an anonymous
referee of a $12$-page extended abstract on the present work for
many helpful comments.

\end{document}